\documentclass[12pt]{article}
\usepackage{amsmath,amsfonts,amssymb}
\usepackage{alltt}

\usepackage{amsmath,amsfonts,ifthen,fullpage,enumerate}  %,amsthm}
\usepackage{mathrsfs}
\usepackage{hyperref}
\usepackage{mathtools}
\usepackage{array}
\usepackage{tikz-cd}
\usepackage{float}
\usetikzlibrary{arrows.meta}
\tikzset{>={Latex[width=2mm,length=2mm]}}

\usepackage{colortbl}

\usepackage{enumerate}

\usepackage{amsmath,amsfonts,amssymb}

\usepackage{makecell}  %thick hline

\usepackage{graphicx}        % standard LaTeX graphics tool
\usepackage{tikz-cd}

\usepackage{tcolorbox}
\def\greyboxvar#1#2
{\begin{tcolorbox} [boxrule=-5pt, boxsep=-0.1cm, colback=white!95!black, 
	width=#2] #1\end{tcolorbox}}

\def\greybox#1
{
	\begingroup
\setlength{\tabcolsep}{2pt} % Default value: 6pt
\renewcommand{\arraystretch}{1.3} % Default value: 1
	\begin{tabular}{c}
\cellcolor{gray!25}#1\end{tabular}
\endgroup
}

\usepackage[all,pdf]{xy}
\usepackage{lmodern,amssymb}   % Drawing lattices

\def\spam{\mathop{\rm span}\nolimits}

\def\trace{\mathop{\rm trace}\nolimits}

\def\rank{\mathop{\rm rank}\nolimits}

\def\pmat#1{\begin{pmatrix}#1\end{pmatrix}}

\def\question#1{{\bf Question: }#1}
\def\question#1{}
\def\cB{{\cal B}}
\def\cG{{\cal G}}

\def\cO{{\cal O}}

\def\R{\mathbb{R}}

\def\CC{\mathbb{C}}
\def\NN{\mathbb{N}}

\def\HH{\mathbb{H}}

\def\OO{\mathbb{O}}

\def\Cd{\C^d}

\def\Hd{\HH^d}

\def\Rd{\R^d}

\def\C{\mathbb{C}}

\newcommand{\RR}{\mathbb{R}}

\newtheorem{example}{Example}[section]

\newtheorem{proposition}{Proposition}[section]

\newtheorem{conjecture}{Conjecture}
\newenvironment{proof}{{\noindent \it
Proof.}}{\hfill$\Box$\medskip}
\newif\ifdraft\def\draft{\drafttrue\hoffset=.8truecm\showlabeltrue
\def\comment##1{{\bf comment: ##1}}
\headline={\sevenrm \hfill \ifx\filenamed\undefined\jobname\else\filenamed\fi%
(.tex) (as of \ifx\updated\undefined???\else\updated\fi)
 \TeX'ed at {\hour\time\divide\hour by 60{}%
\minutes\hour\multiply\minutes by 60{}%
\advance\time by -\minutes
\the\hour:\ifnum\time<10{}0\fi\the\time\  on \today\hfill}}
}

\def\inpro#1{\langle#1\rangle}
\def\ip#1{\langle\kern-.28em\langle#1\rangle\kern-.28em\rangle_\nu}

\def\cO{{\cal O}}

\def\norm#1{\Vert#1\Vert}%\def\norm#1{\|#1\|} does not work with verbatim.tex
\def\openR{{{\rm I}\kern-.16em {\rm R}}}

\let\ga\alpha

\let\gb\beta

\let\gl\lambda

\let\ga\alpha

\let\gb\beta

\def\inpro#1{\langle#1\rangle}

\def\Stab{\mathop{\rm Stab}\nolimits}

\def\Iff{\hskip1em\Longleftrightarrow\hskip1em}
\def\corr{\hskip0.7em\longleftrightarrow\hskip0.7em}
\def\Implies{\hskip1em\Longrightarrow\hskip1em}

\def\formeq{\the\sectionno.\the\equationno}  %% equation numbering
\def\elabel#1/#2/#3/{\global\advance\equationno by 1 %
\ifx#1\empty\else\emember#1%
\ifshowlabel\marginal{\string#1}\fi\fi%
\ifmmode\eqno{#3(\formeq#2)}\else#3\formeq#2\fi} %<------ switch to \leqno ???

\def\makeblanksquare#1#2{
\dimen0=#1pt\advance\dimen0 by -#2pt
      \vrule height#1pt width#2pt depth0pt\kern-#2pt
      \vrule height#1pt width#1pt depth-\dimen0 \kern-#1pt
      \vrule height#2pt width#1pt depth0pt \kern-#2pt
      \vrule height#1pt width#2pt depth0pt
}
\title{\bf 
The quaternionic reflection groups of type $P$ 
and their action on lines in $\HH^2$
}

\title{\bf
Quaternionic MUBs in $\HH^2$ and their reflection symmetries
}

\author{
	Zachary Buckley, Shayne Waldron\\ 
 \\
Department of Mathematics \\ University of Auckland\\
Private
Bag 92019, Auckland, New Zealand\\
e--mail: waldron@math.auckland.ac.nz}

\begin{document}

\maketitle 

\begin{abstract}

We consider the primitive quaternionic reflection groups of type P for $\HH^2$ that are obtained
from Blichfeldt's collineation groups for $\CC^4$.
These are seen to be intimately related to the maximal set of five quaternionic mutually unbiased bases (MUBs) in 
$\HH^2$, for which they are symmetries. 
From these groups, we construct other interesting sets of lines that they fix,
including a new quaternionic spherical $3$-design of $16$ lines in $\HH^2$ 
with angles $\{{1\over5},{3\over5}\}$, which meets the special bound. 
%One suprising consequence of this investigation, 
%is the construction of a new primitive quaternionic reflection group,
%which shows that the classification of these groups is incomplete. 
Some interesting consequences of this investigation include finding imprimitive 
quaternionic reflection groups with several systems of imprimitivity,
and finding a nontrivial reducible subgroup which 
has a continuous family of eigenvectors.

\end{abstract}

\bigskip
\vfill


 \ifdraft (Other designs, configurations) \else\fi
15B33, \ifdraft (Matrices over special rings (quaternions, finite fields, etc.) \else\fi
20C25, \ifdraft Projective representations and multipliers \else\fi
20F55, \ifdraft Reflection and Coxeter groups (group-theoretic aspects) \else\fi
20G20, \ifdraft Linear algebraic groups over the reals, the complexes, the quaternions \else\fi
51F15, \ifdraft Reflection groups, reflection geometries \else\fi
%51M05, \ifdraft (Euclidean geometries (general) and generalizations) \else\fi
%42C15, \ifdraft General harmonic expansions, frames  \else\fi
\quad
secondary
%15B57, \ifdraft (Hermitian, skew-Hermitian, and related matrices) \else\fi
%51E99, \ifdraft Geometry,  None of the above, but in this section \else\fi
%51M15, \ifdraft (Geometric constructions in real or complex geometry) \else\fi
51M20, \ifdraft (Polyhedra and polytopes; regular figures, division of spaces [See also 51F15]) \else\fi
65D30. \ifdraft (Numerical integration) \else\fi
%94A12. \ifdraft (Signal theory [characterization, reconstruction, etc.]) \else\fi

\vskip .5 truecm
\hrule
\newpage

\section{Introduction}

The (finite irreducible) quaternionic reflection groups, i.e., groups of matrices over the quaternions generated 
by reflections, were classified by Cohen \cite{C80}. 
There are six rank two primitive quaternionic reflection groups with primitive 
complexifications, 
in the families O and P, which were obtained from certain collineation groups for $\CC^4$ of Blichfeldt \cite{B17}.
Here we consider the three groups in the family P, and the small sets of quaternionic lines that they stabilise,
which includes the roots of the reflections themselves.

A set of {\bf mutually unbiased bases} (called  {\bf MUBs}) for $\Rd$, $\Cd$ or $\Hd$ is a collection of orthonormal bases
$\cB_1,\ldots,\cB_m$ for which vectors $v$ and $w$ in different bases have a fixed common angle, i.e.,
$$ |\inpro{v,w}|^2 = {1\over d}, \qquad v\in\cB_j, \ w\in\cB_k, \quad j\ne k. $$
Complex MUBs are of interest in quantum information theory as 
they provide unbiased measurements \cite{I81}, \cite{WF89}. 
They are closely related to SICs 
%(see 
\cite{ACFW18}, \cite{W18}. %, \cite{W20}.
%).
The maximal number of MUBs in $\CC^6$ is conjectured to be three \cite{MW24}.

For $d=2$, maximal collections of two and three real and complex MUBs are given by
\begin{equation}
\label{twothreeMUBs}
\bigl\{\pmat{1\cr0},\pmat{0\cr1}\bigr\}, \ \bigl\{{1\over\sqrt{2}}\pmat{1\cr\pm1}\bigr\}, \qquad
\bigl\{\pmat{1\cr0},\pmat{0\cr1}\bigr\}, \ \bigl\{{1\over\sqrt{2}}\pmat{1\cr\pm1}\bigr\}, 
\ \bigl\{{1\over\sqrt{2}}\pmat{1\cr\pm i}\bigr\}.
\end{equation}
There is a maximal set of five quaternionic MUBs in $\HH^2$ given by
\begin{equation}
\label{fiveMUBs}
\bigl\{\pmat{1\cr0},\pmat{0\cr1}\bigr\}, \ \bigl\{{1\over\sqrt{2}}\pmat{1\cr\pm1}\bigr\}, 
\ \bigl\{{1\over\sqrt{2}}\pmat{1\cr\pm i}\bigr\}, \ \bigl\{{1\over\sqrt{2}}\pmat{1\cr\pm j}\bigr\}, 
\ \bigl\{{1\over\sqrt{2}}\pmat{1\cr\pm k}\bigr\}.
\end{equation}
These first appeared in Example 3 of \cite{H82} as 
a ``tight $3$-design attaining the absolute bound''.
The Example 4 then extends this to what one might call ``nine octonionic MUBs in $\OO^2$''.
We will not consider the general theory of (maximal) quaternionic MUBs, other than remarking that it begins
with (\ref{fiveMUBs}), the Example 21 of \cite{H82}, which gives nine quaternionic MUBs in $\HH^4$,
and various quaternionic MUB-like configurations \cite{H82}, \cite{K95}, \cite{CKM16}, \cite{BSDL24}.

The rest of the paper is set out as follows. We first consider Blichfeldt's original collineation groups for $\CC^4$,
and construct from them the primitive reflection groups of the type P in \cite{C80} (where details were not given).
This leads to nice presentations for the groups, which include as an imprimitive reflection group based on the five MUBs
extended by adding a single non-monomial reflection matrix.
%The latter approach leads to a new rank two quaternionic reflection group of 
%order $64$, with $14$ reflections.

Next, we observe that the $P$ groups are symmetries of the MUB lines, and consider the associated
permutation action on these lines and the MUB pairs.
We then consider the roots of the reflections, 
and the sets of lines stabilised by these reflection groups. % in the family P. 
In other words, we recognise the reflection groups of type $P$ as symmetry groups of nice (well spaced) 
configurations of quaternionic lines, such as the five quaternionic MUBs.
In particular, we construct a new spherical $3$-design of $16$ lines in $\HH^2$
with angles $\{{1\over5},{3\over5}\}$, which meets the special bound, and
give a nontrivial reducible subgroup which
has a continuous family of eigenvectors.

\vfil\eject
\section{The quaternionic reflection groups of type P}

%We assume basic familiarity with the classification of complex reflection groups \cite{LT09}.
We assume some basic familiarity with finite irreducible complex reflection groups 
and their classification into those which are primitive and imprimitive \cite{ST54}, \cite{LT09},
and the quaternions $\HH$ and matrices over them, see, e.g., \cite{SS95}, \cite{Z97}, \cite{CS03}, \cite{V21}.

A quaternionic reflection on $\Hd$ is a nonidentity matrix $g\in M_d(\HH)$ 
of finite order which fixes a hyperplane pointwise, equivalently, $\rank(I-g)=d-1$, and a 
finite reflection group is a finite subgroup of $M_d(\HH)$ which is generated
by reflections. Since finite subgroups of $M_d(\HH)$ are conjugate to groups
of unitary matrices, we suppose henceforth that our reflection groups are unitary. 
Therefore, a (unitary) reflection $g$ is defined by a {\bf root vector} $a\in\Hd$,
and a unit scalar $\xi\in\HH$, $\xi\ne1$, for which
$$ g x = x, \quad\forall x\in a^\perp, \qquad g a = a\xi, $$
i.e., the fixed hyperplane is the orthogonal complement of $a$, 
and $g$ has order $n$ if and only $\xi$ is a primitive $n$-th root of unity.
%If $a$ is a unit vector, 
If $a\ne 0$ is a vector, 
then a formula for $g$ is
\begin{equation}
\label{raxiform}
r_{a,\xi} := I-{a(1-\xi)a^*\over\inpro{a,a}}.
\end{equation}
Throughout, $\Hd$ is considered as a right vector space ($\HH$-module), 
so that linear maps are applied on the left,
and we denote the quaternion group by
$$ Q_8:=\{1,-1,i,-i,j,-j,k,-k\}. $$

\begin{example}
By (\ref{raxiform}), the reflection for the root
	$$ a={1\over\sqrt{2}}\pmat{1\cr-b^{-1}}\in\HH^2, \quad |b|=1, $$
is 
$$ r_{a,\xi} %=I-{1\over2}\pmat{1\cr-b^{-1}}(1-\xi)\pmat{1&-b}
= {1\over2}\pmat{1+\xi &(1-\xi)b\cr b^{-1}(1-\xi)&b^{-1}(1+\xi)b}, $$
which is monomial if and only if $\xi=-1$, which gives the reflection of order two
$$ r_{a,-1}=\pmat{0&b\cr b^{-1}&0}. $$
Therefore, the reflections of order two given by the ten MUB vectors of (\ref{fiveMUBs})
	are 
\begin{equation}
\label{MUBreflections}
\pmat{-1&0\cr0&1}, \pmat{1&0\cr0&-1}, \quad
\pm\pmat{0&1\cr 1&0}, \quad
\pm\pmat{0&i\cr -i&0}, \quad
\pm\pmat{0&j\cr -j&0}, \quad
\pm\pmat{0&k\cr -k&0}.
\end{equation}
\end{example}

A reflection group $G$ is said to be {\bf imprimitive} if the space on which it acts
can be decomposed into proper subspaces which it permutes, a so called system of imprimitivity,
otherwise it is said to be {\bf primitive}. For a reflection group (which is unitary)
acting on $\HH^2$, we can assume the system of imprimitivity is given by the standard basis vectors
$$ V_1=\spam_\HH e_1, \qquad  V_2=\spam_\HH e_2, $$
so that its elements have the (monomial) form
$$ \pmat{a&0\cr0&d}, \quad \pmat{0&b\cr c&0}, \qquad |a|=|b|=|c|=|d|=1. $$

Blichfeldt \cite{B17} (Chapter VII) classified the irreducible collineation groups for $\CC^4$.
Collineation groups are groups of matrices defined up to a scalar multiple 
(in modern terminology linear maps defined on projective spaces). 
The groups (A), (C), (K) of \cite{B17} have orders $60\phi$, $360\phi$, $720\phi$ 
(here $\phi$ indicates the order of the subgroup of scalar matrices,
which is unimportant in the theory), and correspond to the groups of type O in \cite{C80} of
orders $120$, $720$, $1440$ (see \cite{W24}). 

Here we consider the groups $14^\circ$, $16^\circ$, $18^\circ$ of orders
$10\cdot 16\phi$, $60\cdot 16\phi$, $120\cdot 16\phi$, 
which correspond to the groups of type P in \cite{C80} of orders $320$, $1920$, $3840$.
The collineation groups $13^\circ,\ldots,21^\circ$ of \cite{B17} (page 172) are primitive
collineation groups generated by an imprimitive group $K$ of order $16\phi$, generated by
%$$ A_1=\pmat{1&0&0&0 \cr 0&1&0&0 \cr 0&0&-1&0 \cr 0&0&0&-1 }, \quad
%A_2=\pmat{1&0&0&0 \cr 0&-1&0&0 \cr 0&0&-1&0 \cr 0&0&0&1 }, \quad
%A_3=\pmat{0&1&0&0 \cr 1&0&0&0 \cr 0&0&0&1 \cr 0&0&1&0 }, \quad
%A_4=\pmat{0&0&1&0 \cr 0&0&0&1 \cr 1&0&0&0 \cr 0&1&0&0 }, $$
$$ A_1=\pmat{1&0&0&0 \cr 0&1&0&0 \cr 0&0&-1&0 \cr 0&0&0&-1 }, \quad
A_2=\pmat{1&0&0&0 \cr 0&-1&0&0 \cr 0&0&-1&0 \cr 0&0&0&1 }, $$
$$ A_3=\pmat{0&1&0&0 \cr 1&0&0&0 \cr 0&0&0&1 \cr 0&0&1&0 }, \quad
A_4=\pmat{0&0&1&0 \cr 0&0&0&1 \cr 1&0&0&0 \cr 0&1&0&0 }, $$
%$$ A_1,A_2,A_3,A_4=\pmat{1&0&0&0 \cr 0&1&0&0 \cr 0&0&-1&0 \cr 0&0&0&-1 }, \
%\pmat{1&0&0&0 \cr 0&-1&0&0 \cr 0&0&-1&0 \cr 0&0&0&1 }, \
%\pmat{0&1&0&0 \cr 1&0&0&0 \cr 0&0&0&1 \cr 0&0&1&0 }, \
%\pmat{0&0&1&0 \cr 0&0&0&1 \cr 1&0&0&0 \cr 0&1&0&0 }, $$
the element $T$, %(see below), 
which gives the primitive group $13^\circ$, and various other elements given by
$$ S={1+i\over\sqrt{2}}\pmat{i&0&0&0\cr 0&i&0&0\cr 0&0&1&0\cr 0&0&0&1}, \qquad
T={1+i\over2}\pmat{-i&0&0&i\cr 0&1&1&0\cr 1&0&0&1\cr 0&-i&i&0}, $$
$$ R={1\over\sqrt{2}}\pmat{1&i&0&0\cr i&1&0&0\cr 0&0&i&-1\cr 0&0&1&-i}, \quad
A={1+i\over\sqrt{2}}\pmat{1&0&0&0\cr 0&i&0&0\cr 0&0&i&0\cr 0&0&0&1}, \quad
B={1+i\over\sqrt{2}}\pmat{1&0&0&0\cr 0&1&0&0\cr 0&0&1&0\cr 0&0&0&-1}. $$
These groups, which have $K$ as a normal subgroup, are generated as follows
\begin{align}
\label{13-21defs}
& 13^\circ:\ K,T, \qquad\qquad 14^\circ:\ K,T,R^2, \qquad\ 15^\circ:\ K,T,R, \cr
& 16^\circ:\ K,T,SB, \qquad 17^\circ:\ K,T,BR, \qquad 18^\circ:\ K,T,A, \cr
& 19^\circ:\ K,T,B, \qquad\ \ 20^\circ:\ K,T,AB, \qquad 21^\circ:\ K,T,S.
\end{align}

A group $G$ of complex matrices in $M_{2d}(\CC)$ gives rise to 
a group of quaternionic matrices in $M_d(\HH)$ if it
is conjugate to a group of matrices of the %following
{\bf symplectic  form}
\begin{equation}
\label{C2dHdcorres}
\pmat{A&-B\cr\overline{B}&\overline{A}}\in M_{2d}(\CC)
\Iff A+Bj  \in M_d(\HH).
\end{equation}
%$$ g=A+Bj \Iff [g]_\CC = \pmat{A&-B\cr\overline{B}&\overline{A}}, $$
The {\bf Frobenius-Schur indicator} of a complex representation of a finite group $G$ is
$$ \iota\chi %\gs_\CC(G)
:= {1\over|G|}\sum_{g\in G} \chi(g^2)\in\{-1,0,1\}, $$
where $\chi$ is the character of the representation.
This takes the value $-1$ if and only if the representation 
of $G$ corresponds to a quaternionic representation via
(\ref{C2dHdcorres}) \cite{G11}. 
Blichfeldt's collineation groups $13^\circ,\ldots,21^\circ$ 
are given as matrix groups over $\CC$, with a subgroup of scalar matrices (of order $\phi$). Changing the 
subgroup of scalar matrices (which gives the same collineation group), 
changes the Frobenius-Schur indicator, and so some care must be taken. 
Indeed, in view of (\ref{C2dHdcorres}), the collineation group must be
presented so that its matrices have a real trace,
and consequently $\pm I\in M_{2d}(\CC)$ are the only allowable scalar matrices.
It is still quite possible for the group of quaternionic matrices to contain
scalar matrices, e.g., $iI,jI\in M_d(\HH)$ correspond via (\ref{C2dHdcorres}) to
the nonscalar symplectic matrices
$$ \pmat{iI&0\cr0&-iI}, 
\pmat{0&-I\cr I& 0}\in M_{2d}(\CC). $$

We first consider the group $K$, which is an imprimitive normal subgroup
of all the collineation groups, together with $T$. 
The group generated by $A_1,A_2,A_3,A_4$ has small group identifier
$\inpro{32, 49}$ and has Frobenius-Schur indicator $1$, and so 
does not correspond to a group in $M_2(\HH)$.
The trace of $T$ is ${1}$, but it is not in the symplectic form (\ref{C2dHdcorres}). 
Conjugation of $T$ (equivalently $-T$) by the permutation matrices for the permutations
$(1\, 4)$, $(2\, 3)$, $(1\, 3\, 4\, 2)$, $(1\, 2\, 4\, 3)$ gives a 
matrix of the form (\ref{C2dHdcorres}). Calculations show that whatever permutation
is taken, the quaternionic groups obtained are identical elementwise (just with different
generators). We take
$$ P=P_{(1\, 4)} =\pmat{0&0&0&1\cr 0&1&0&0\cr 0&0&1&0\cr 1&0&0&0}, $$
which conjugates $A_1,iA_2,iA_3,A_4$ to the symplectic form (\ref{C2dHdcorres}), i.e.,
\begin{equation}
\label{KHgenerators}
\pmat{-1&0\cr0&1},\quad \pmat{i&0\cr0&-i},\quad
\pmat{-k&0\cr0&-k}, \quad \pmat{0&1\cr1&0}.
\end{equation}
The group generated by $A_1,iA_2,iA_3,A_4$ has identifier $\inpro{32, 50}$, and
Schur-Frobenius indicator $-1$. Here $A_2,A_3$, which have zero trace, were
multiplied by $i$ to obtain the symplectic form. The group $K$ generated by
the quaternionic matrices of (\ref{KHgenerators}) is
the imprimitive reflection group generated by 
the ten MUB reflections of (\ref{MUBreflections}),
which are all of its reflections. This is the group denoted by 
$$ K=G_{Q_8}(Q_8,C_2) =\cG(\{1,i,j,k\},\{\}), \qquad 
K/\inpro{-I}=C_2\times C_2\times C_2\times C_2,$$
%$$ K=G_{Q_8}(Q_8,C_2) =\cG(\{1,i,j,k\},\{\}), $$
in \cite{W25}, which is generated by the reflections
\begin{equation}
\label{Krefgens}
\pmat{0&1\cr1&0},\quad \pmat{0&i\cr-i&0},\quad \pmat{0&j\cr-j&0},\quad \pmat{0&k\cr-k&0}.
\end{equation}
The $32$ elements of $K$ are
$$ \pmat{q&0\cr0&\pm q}, \quad \pmat{0&q\cr\pm q&0}, \qquad q\in Q_8. $$

Despite four of the reflection pairs for the MUBs (\ref{MUBreflections}) being nondiagonal matrices, with
the other being diagonal, each pair plays an equivalent role. 
For example, the orbits of the reflections under conjugation by $K$ are of size two,
giving the MUB pairs of reflections, and the minimal generating sets of reflection 
for $K$ are precisely any four reflections, which come from different MUB pairs (orbits).
It is also interesting to note that each of the five MUBs of (\ref{fiveMUBs}) 
are a system of imprimitivity for $K$. 
By way of comparison
(see \cite{LT09} Theorem 2.16), the only complex reflection groups on $\RR^2$ or $\CC^2$ 
which have more than one system of imprimitivity are $G(2,1,2)\cong G(4,4,2)$ and
$G(4,2,2)$, which
have three systems of imprimitivity given by the three complex MUBs of (\ref{twothreeMUBs}).

The conjugate of $T$ by $P$ gives a symplectic matrix of order ten, i.e.,
\begin{equation}
\label{Tmat}
	t:={1\over 2}\pmat{j-k& 1-i\cr-j-k&1+i},
\end{equation}
which is not a reflection. It maps the lines given by the MUB vectors to themselves, 
e.g.,
$$ \pmat{1\cr0} 
\mapsto \pmat{1\cr -i}
\mapsto \pmat{1\cr k}
\mapsto \pmat{1\cr -1}
\mapsto \pmat{1\cr j}
\mapsto \pmat{1\cr 0}. $$
The group of quaternionic matrices generated by $K$, i.e., 
the reflections (\ref{Krefgens}), and $t$ of (\ref{Tmat}), contains no further reflections, and hence is not a reflection group.

%The groups of Blichfeldt \cite{B17} (page 172) are generated as follows
%$$ 13^\circ:\ K,T, \quad 14^\circ:\ K,T,R^2, \quad 
%15^\circ:\ K,T,R, \quad 16^\circ:\ K,T,SB,  $$
%$$ 17^\circ:\ K,T,BR, \quad 18^\circ:\ K,T,A, \quad 
%19^\circ:\ K,T,B, \quad 20^\circ:\ K,T,AB, \quad
%21^\circ:\ K,T,S, $$
The other generators from (\ref{13-21defs}) which have real traces and 
conjugate under $P$ to a symplectic matrix are $R^2,R,SB$, and those
which do not have real traces are
$$ 
\trace(BR)=2i, \quad
\trace(A)=\trace(S)=2\sqrt{2}i, \quad
\trace(B)=\sqrt{2}(1+i). $$
After scaling to obtain a real trace, 
only $iA$ conjugates under $P$ to the symplectic form.
After conjugation with $P$, the matrices $R^2$, $R$, $SB$, $iA$ 
are in the symplectic form, giving
\begin{equation}
\label{othergens}
\pmat{-1&0\cr 0&-k}, \quad
{1\over\sqrt{2}}\pmat{-i-j&0\cr0&1-k}, \quad
\pmat{-i&0\cr0&-1},\quad
{1\over\sqrt{2}}\pmat{-1+i&0\cr0&-1-i}.
\end{equation}
Hence there are primitive quaternionic groups of matrices in 
$U_2(\HH)$ corresponding to Blichfeldt's groups 
$13^\circ,\ldots,16^\circ$, $18^\circ$, which are generated by the 
corresponding matrices from (\ref{Krefgens}), (\ref{Tmat}), (\ref{othergens}).
A calculation in {\tt magma} shows that three of these are reflection groups, 
i.e.,
$$ G_{14^\circ}=\inpro{K,t,\pmat{-1&0\cr 0&-k}}, \quad
G_{16^\circ}=\inpro{K,t,\pmat{-i&0\cr0&-1}}, \quad
G_{18^\circ}=\inpro{K,t,{1\over\sqrt{2}}\pmat{-1+i&0\cr0&-1-i}}. $$
Since $K$ is contained in these groups, adding the last generator is the
same as adding the reflections
\begin{equation}
\label{reflectionsadded}
\pmat{k&0\cr 0&1}, \quad \pmat{i&0\cr0&1}, \quad 
{1\over\sqrt{2}}\pmat{0&-1+i\cr-1-i&0} = r_{(1,{1+i\over\sqrt{2}}),-1},
\end{equation}
respectively, and since
$$ t\pmat{k&0\cr 0&1} = {1\over2}\pmat{1+i&1-i\cr 1-i &1+i}=r_{(1,-1),i}, $$
is a reflection (of order $4$), we can conclude (by hand) that the above are 
reflection groups, with
\begin{equation}
\label{G14naivedef}
G_{14^\circ} = \inpro{K,r_{(1,0),k},r_{(1,1),i}},
\end{equation}
since $r_{(1,1),i}\in G_{14^\circ}$. 
The $30$ reflections in $G_{14^\circ}$ are given by the $\xi,a$ pairs
\begin{equation}
\label{G14refs}
\inpro{i}:\quad \pmat{1\cr\pm1}, \ \pmat{1\cr\pm k}, \quad
 \inpro{j}:\quad \pmat{1\cr\pm i}, \ \pmat{1\cr\pm j}, \quad
 \inpro{k}:\quad \pmat{1\cr0},\ \pmat{0\cr1},
\end{equation}
and the $70$ reflections in $G_{16^\circ}$ by
\begin{equation}
\label{70reflections}
Q_8:\quad \pmat{1\cr0},\ \pmat{0\cr1},\  \pmat{1\cr\pm1}, \ \pmat{1\cr\pm i}, \
\pmat{1\cr\pm j}, \ \pmat{1\cr\pm k}.
\end{equation}
We observe from above that $G_{14^\circ}$ is a subgroup of $G_{16^\circ}$.
A calculation shows that there are six conjugates
$G_{q_1,q_2}$, $q_1\ne q_2$, $q_1,q_2\in\{i,j,k\}$, of $G_{14^\circ}$ in $G_{16^\circ}$,
given by the reflections
\begin{equation}
\label{refsMUBform}
\inpro{q_1}:\quad \pmat{1\cr0},\ \pmat{0\cr1}, \quad
\inpro{q_2}:\quad \pmat{1\cr\pm 1}, \ \pmat{1\cr\pm q_1}, \quad
\inpro{q_1 q_2}:\quad \pmat{1\cr\pm q_2}, \ \pmat{1\cr\pm q_1 q_2}.
\end{equation}
In view of (\ref{G14naivedef}),
for $G_{q_1,q_2}$ we can take any generators for $K$, together with
the reflections $r_{(1,0),q_1}$, $r_{(1,1),q_2}$. It turns out these generators
alone are sufficient, i.e.,
\begin{equation}
\label{Pq1q2nice}
G_{q_1,q_2} = \inpro{r_{(1,0),q_1},r_{(1,1),q_2}}
= \inpro{\pmat{q_1&0\cr0&1},
{1\over2} \pmat{1+q_2&-1+q_2\cr-1+q_2&1+q_2}}.
\end{equation}

The scalar $\xi$ for a reflection $r_{a,\xi}$ depends on the
particular multiple of the root taken, i.e.,
$$ r_{a,\xi} = r_{a\gb,\gb^{-1}\xi\gb}, \qquad\gb\in\HH^*. $$
In \cite{C80}, the group $H_a$ of scalars associated with a root is 
taken to be the same for all roots $a$. To change the scalars $\xi=q_2$ to $q_1$
and $\xi=q_1q_2$ to $q_1$ in (\ref{refsMUBform}), we take 
$\gb= (1-q_1 q_2)$ and $\gb= (1-q_2)$, to obtain
$$ 
\inpro{q_1}:\quad \pmat{1\cr0},\ \pmat{0\cr1},\
\pmat{1\cr\pm 1}(1-q_1 q_2), \ \pmat{1\cr\pm q_1}(1-q_1 q_2),\
\pmat{1\cr\pm q_2}(1-q_2), \
\pmat{1\cr\pm q_1 q_2}(1-q_2). $$
Taking $q_1=j$, $q_2=k$ above gives the root system of \cite{C80} (Table II),
and so we conclude the group given there is $G_{j,k}$, whereas the group
given by (\ref{G14refs}) is $G_{k,i}$.

We now consider generators for the groups of type $P$.
%, i.e., 
For the first $G_{14^\circ}$, % $G_{16^\circ}$, $G_{18^\circ}$.
(\ref{Pq1q2nice}) gives
%In view of (\ref{G14naivedef}),
%for $G_{q_1,q_2}$ we can take any generators for $K$, together with
%the reflections $r_{(1,0),q_1}$, $r_{(1,1),q_2}$. It turns out these generators 
%alone are sufficient. Thus we have the following presentation
%for the first of the groups of type $P$
\begin{equation}
\label{P1nice}
P_1=H_{320}:=G_{i,j} = \inpro{r_{(1,0),i},r_{(1,1),j}}
= \inpro{\pmat{i&0\cr0&1}, 
{1\over2} \pmat{1+j&-1+j\cr-1+j&1+j}}, 
\end{equation}
% $$\pmat{{1+j\over2}&{-1+j\over2}\cr{-1+j\over2}&{1+j\over2}} $$
where $|P_1|=320$. The comment of (\ref{reflectionsadded}) implies that the next group 
$G_{16^\circ}$ is
\begin{equation}
\label{P2nice}
P_2=H_{1920}:=\inpro{P_1,r_{(1,0),j}}
= \inpro{ \pmat{i&0\cr0&1}, \pmat{j&0\cr0&1}, {1\over2} \pmat{1+j&-1+j\cr-1+j&1+j}}, 
\end{equation}
where $|P_2|=1920$. Both of these groups have the five MUBs as 
the roots of their reflections. Similar considerations give $G_{18^\circ}$ as
$$ G_{18^\circ}=\inpro{P_2,r_{(1,{1+i\over\sqrt{2}}),-1}}
= \inpro{ P_2, {1\over\sqrt{2}} \pmat{0&-1+i\cr-1-i&0}}. $$
It is easily verified that $G_{18^\circ}$ contains the reflection 
given by the Fourier matrix, i.e.,
\begin{equation}
\label{Ftransform}
%r_{(1-\sqrt{2},1),-1} =
F:=r_{(-1,1+\sqrt{2}),-1} =
{1\over\sqrt{2}}\pmat{1&1\cr1&-1}\in G_{18^\circ},
\end{equation}
which leads to the generating reflections
\begin{equation}
\label{P3nice}
P_3= H_{3840} 
:= \inpro{ \pmat{i&0\cr0&1},\pmat{j&0\cr0&1},{1\over\sqrt{2}}\pmat{1&1\cr1&-1}},
%P_3:=\inpro{r_{(1,0),i},r_{(1,0),j},r_{({1\over\sqrt{2}},{1+i\over2}),-1}}
%= \inpro{ \pmat{i&0\cr0&1}, \pmat{j&0\cr0&1}, {1\over\sqrt{2}} \pmat{-1+i&0\cr0&-1-i}}, 
\end{equation}
where $|P_3|=3840$.

\begin{figure}[h!]
\begin{center}
\begin{tikzpicture}
    \matrix (A) [matrix of nodes, row sep=1.0cm, column sep = 0.1 cm]
    {
	    &&&&& $P_3$ &&&&& \\
	    &&&&& $P_2$ &&&&& \\
	    $P_1$ && $P_1$ && $P_1$ && $P_1$ && $P_1$ && $P_1$ \\
	    &&&&& $K$ &&&&& \\
    };
    \draw (A-1-6)--(A-2-6);
    \draw (A-2-6)--(A-3-1);
    \draw (A-2-6)--(A-3-3);
    \draw (A-2-6)--(A-3-5);
    \draw (A-2-6)--(A-3-7);
    \draw (A-2-6)--(A-3-9);
    \draw (A-2-6)--(A-3-11);
    \draw (A-3-1)--(A-4-6);
    \draw (A-3-3)--(A-4-6);
    \draw (A-3-5)--(A-4-6);
    \draw (A-3-7)--(A-4-6);
    \draw (A-3-9)--(A-4-6);
    \draw (A-3-11)--(A-4-6);
\end{tikzpicture}
\end{center}
\vskip-0.7truecm
	\caption{The $P$ groups: $K\lhd P_1,P_2,P_3$ and $P_2\lhd P_3$.
	The group $P_1$ occurs six times 
	as a subgroup of $P_2$ (a single conjugacy class), i.e., as
	$G_{q_1,q_2}$, $q_1\ne q_2$, $q_1,q_2\in\{i,j,k\}$. }
\end{figure}
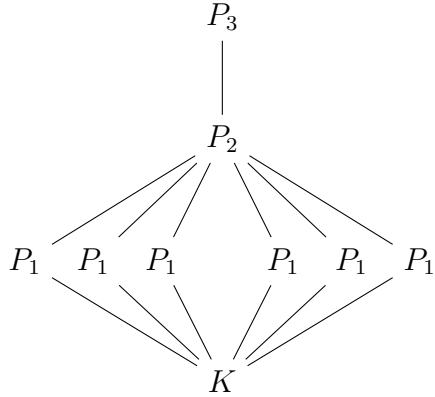

The reflection group $P_3$ has a $110$ reflections, consisting of the 
$70$ reflections (\ref{70reflections}) of $P_2$, and $40$ reflections 
of order two which are the orbit of $F$ under the conjugation action of $P_2$.
These are given by the root lines
%The $40=24+2\cdot 8$ reflections in the orbit of $(-1,1+\sqrt{2})$ are
\begin{equation}
\label{40rootlinesP3}
\pmat{\sqrt{2}\cr p+q},\ p+q\ne0, \quad \{p,q\}\subset Q_8, \qquad
\pmat{1+\sqrt{2}\cr q},\ \pmat{q\cr1+\sqrt{2}},  \quad q\in Q_8.
\end{equation}
The first $24$ of these give rise to monomial reflections, i.e.,
%$$ \pmat{0&{p+q\over\sqrt{2}}\cr({p+q\over\sqrt{2}})^{-1}&0},\quad 
%p+q\ne0, \quad \{p,q\}\subset Q_8. $$
\begin{equation}
\label{bform}
\pmat{0&b \cr b^{-1}&0},\quad 
b={p+q\over\sqrt{2}}\ne0, \quad \{p,q\}\subset Q_8,
\end{equation}
and the last $16$ give rise to the non-monomial reflections
\begin{equation}
\label{nonmomform}
{1\over\sqrt{2}}\pmat{-1&-\overline{q}\cr-q&1}, \,
 {1\over\sqrt{2}}\pmat{1&-q\cr-\overline{q}&-1}, \quad q\in Q_8.
\end{equation}

\begin{example}
\label{fiveorder2}
It is interesting to observe that $P_3$ is generated by five of the 
$40$ reflections of order two given by (\ref{40rootlinesP3}), e.g.,
$$ \pmat{0&b\cr b^{-1}&0}, \quad \sqrt{2}b\in\{1+i,1-i,1+j,1+k\}, \qquad
{1\over\sqrt{2}}\pmat{1&1\cr1&-1}. $$
This is a consequence of $P_3$ being given in \cite{C80} via a root
system based on %the root lines of 
(\ref{40rootlinesP3}).
\end{example}

\begin{table}[!h]
\caption{Generating reflections for $K$ %(take any four) 
and the $P$ groups.
The group $P_3$ is also generated by five of the $40$ reflections of order two 
given by (\ref{bform}) and (\ref{nonmomform}) (see Example \ref{fiveorder2}). }
%\vskip0.3truecm
\label{PKgenrefs}
\begin{center}
\begin{tabular}{ |  >{$}l<{$} | >{$}l<{$} |} 
\hline
        &\\[-0.3cm]
G 
	& \hbox{generating reflections} 
	\\[0.1cm]
\hline
&\\[-0.3cm]
K^\dagger & \pmat{-1&0\cr0&1}, \pmat{0&1\cr1&0},\pmat{0&i\cr-i&0}, \pmat{0&j\cr-j&0},
	\pmat{0&k\cr-k&0} \\[0.4cm]
	P_1 & \pmat{i&0\cr0&1}, {\displaystyle {1\over2} \pmat{1+j&-1+j\cr-1+j&1+j}} \\[0.4cm]
	P_2 & \pmat{i&0\cr0&1}, \pmat{j&0\cr0&1}, {\displaystyle {1\over2} \pmat{1+j&-1+j\cr-1+j&1+j} } \\[0.4cm]
	P_3 & \pmat{i&0\cr0&1},\pmat{j&0\cr0&1}, {\displaystyle {1\over\sqrt{2}}\pmat{1&1\cr1&-1} } \\[0.4cm]
%&&&&&&&\\[-0.3cm]
\hline
	\multicolumn{2}{l}{\hbox{\footnotesize ${}^\dagger$ $K$ is generated by any four of these reflections. }}

%\multicolumn{7}{l}{\hbox{\footnotesize $\dagger$ Cases where the stabiliser $H$ of the fiducial
%vector in $P_j$ is a maximal reducible subgroup of $G=P_j$.  } }
\end{tabular}
%\multicolumn{7}{l}{\hbox{\footnotesize The above fiducial vectors are for the reflection group with complex reflection group as
%a subgroup. To find the corresponding fiducial vectors for Cohen's presentation of the
%group multiply the vectors by $u$.  } }
\end{center}
\end{table}

\begin{example} 
\label{newrefgroup}
Consider the reflection group of order $64$ given by 
$$ P_0=\inpro{ \pmat{0&1\cr1&0}, \pmat{0&i\cr-i&0}, \pmat{0&j\cr-j&0}, 
{1\over\sqrt{2}}\pmat{1&1\cr1&-1} }, $$
which is generated by $K$ and the Fourier matrix $F$ of (\ref{Ftransform}).
This was initially assumed to be primitive, 
and hence a previously unknown such group, 
but conjugation by the matrices
$$ {1\over\sqrt{2}}\pmat{1&i\cr i&1},
	{1\over\sqrt{2}}\pmat{1&j\cr j&1},
	{1\over\sqrt{2}}\pmat{1&k\cr k&1}
\in P_3 $$ 
gives a monomial, and hence
imprimitive, reflection group, which has three systems of imprimitivity.
It is the imprimitive reflection group $G(4,1,2,2)$ 
in the family of \cite{W25}.
\end{example}

There are only four imprimitive complex reflection groups with more than one
system of imprimitivity (Theorem 2.16 \cite{LT09}), i.e., two in $\CC^2$ 
(three systems), one in $\CC^3$ (four systems), and one in $\CC^4$ 
(three systems).
Therefore, given the above example (three systems) and $K$ (five systems), 
it appears that the number of systems of imprimitivity for
quaternionic reflection groups is worthy of some study.

%Given this example (three systems of imprimitivity), and $K$ which has five
%systems of imprimitivityexample of $K$
%On the basis of its order, 
%it does not seem to appear in Cohen's classification \cite{C80}.
%	It orbit type is $2C_2,2C_2,2C_2,4C_2,4C_2$ ($14$ reflections).

\section{The maximal imprimitive reflection subgroups}

Our presentations (\ref{P1nice}), (\ref{P2nice}), (\ref{P3nice}) 
of the $P$ reflection groups (see Table \ref{PKgenrefs}) involve just a single non-monomial reflection.
Therefore, we can view them as imprimitive (monomial) reflection groups 
with a single non-monomial reflection added. We now identify these imprimitive 
reflection groups.

The non-diagonal reflections in a reflection group of rank two
have the form (\ref{bform}) for a set $L$ of $b\in\HH$,
called a ``reflection system'' \cite{W25},
which satisfy the conditions
\begin{enumerate}
\item $L$ generates a finite group $K=\inpro{L}$.
\item $L$ is closed under the binary operation $(a,b)\mapsto a \circ b := ab^{-1}a$.
\item $1\in L$.
%(this is a technical condition to ensure the corresponding reflection is in $G$).
\end{enumerate}
If $L$ is the closure of $X\subset\HH$ under $(a,b)\mapsto a \circ b$,
then we say $X$ generates the reflection system $L$, and we write $L=L(X)$.
In view of (\ref{MUBreflections}), (\ref{bform}), the reflection systems for the groups 
$P_1$, $P_2$ are $Q_8=L(\{1,i,j,k\})$, and for $P_3$ it is % the $32$ elements
$$ L_{32}:= L\bigl(\bigl\{1,{1+i\over\sqrt{2}},{1+j\over\sqrt{2}},{1+k\over\sqrt{2}}\bigr\}\bigr)
=Q_8 \cup \bigl\{{p+q\over\sqrt{2}}\ne0 : \{p,q\}\subset Q_8\bigr\}. $$
%$$ Q_8 \cup \bigl\{\hbox{${p+q\over\sqrt{2}}\ne0$} :\{p,q\}\subset Q_8\bigr\}. $$
The $32$-element reflection system $L_{32}$ has $K$ the binary octahedral group of
order $48$ given by
$$ \cO = \inpro{{1+i\over\sqrt{2}},{1+i+j+k\over2}}
=\inpro{{1+i\over\sqrt{2}},{1+j\over\sqrt{2}} }. $$
It is equivalent to that given in \cite{W25}, i.e.,
$$ L_{32} ={1+i\over\sqrt{2}} L_{32}^\cO, \qquad
L_{32}^\cO:=L\bigl(\bigl\{1,{1+i\over\sqrt{2}},{1+i+j+k\over2}\bigr\}\bigr). $$
The corresponding subreflection systems are
$$ 
L_{20} = L\bigl(\bigl\{1,{1+i\over\sqrt{2}},{1+j\over\sqrt{2}},k\bigr\}\bigr), \quad
L_{18} = L\bigl(\bigl\{1,{1+i\over\sqrt{2}},{1+j\over\sqrt{2}}\bigr\}\bigr), \quad
L_{14} = L\bigl(\bigl\{1,{i-j\over\sqrt{2}},{i-k\over\sqrt{2}},{j+k\over\sqrt{2}}\bigr\}\bigr).
$$

The imprimitive quaternionic reflection groups of rank two have a canonical form
$G=G_K(L,H)$, where $L$ is a reflection system with $K=\inpro{L}$, which 
gives the nondiagonal reflections, and $H$ is normal subgroup of $K$, 
which gives the diagonal reflections in $G$, i.e.,
$$ \pmat{h&0\cr0&1}, \, \pmat{1&0\cr0&h}, \qquad h\in H,\, h\ne1. $$
The monomial reflections of $G=P_1,P_2,P_3$ generate the following
imprimitive reflection groups $G_M$
%The imprimitive reflection groups $G_M$ generated by 
%the monomial reflections of $G=P_1,P_2,P_3$ $P_1$, $P_2$, $P_3$ are
$$ G_{Q_8}(Q_8,C_4), \quad  G_{Q_8}(Q_8,Q_8), \quad
G_\cO(L_{32}^\cO,Q_8), $$
which have orders $64$, $128$, $768$.

For a reflection group $G$, the reflections for a given root $a$ together with
the identity form a subgroup $R_a$.
The group $G$ acts on the reflection subgroups $R_a$ via conjugation.
We will refer to the orbits of this action as the {\bf reflection type}
or {\bf reflection orbits} of $G$. If the $m$ reflection orbits are given by 
$R_{a_1},\ldots,R_{a_m}$, we will often write the reflection type as
$n_1 R_{a_1},\ldots,n_m R_{a_m}$, where $n_j$ is the orbit size, and $R_{a_j}$ 
is an abstract group.

We can now summarise the structure of the $P$ groups.

\begin{table}[!h]
	\caption{The $P$ groups and their monomial reflection subgroup $G_M$ 
	(each of these appears five times, corresponding to the five sets of imprimitivity for $K$). }
        \vskip0.3truecm
\label{TOIgroups-table}
\begin{tabular}{ |  >{$}l<{$} | >{$}c<{$} | >{$}l<{$} | >{$}l<{$} | >{$}l<{$} | >{$}l<{$} | >{$}l<{$} | >{$}l<{$} |}
\hline
        &&&&&&&\\[-0.3cm]
G & |G| & \hbox{refs} & \hbox{ref orbits} & G_M &
|G_M| &\hbox{refs} & \hbox{ref orbits} \\[0.1cm]
\hline
&&&&&&&\\[-0.3cm]
P_3 & 3840 & 110 & 10Q_8, 40C_2 & G_\cO(L_{32}^\cO,Q_8) & 768 & 46 & 2Q_8, 8C_2, 24C_2\\ 
P_2 & 1920 & 70 & 10Q_8 & G_{Q_8}(Q_8,Q_8) & 128 & 22 & 2Q_8, 8C_2  \\
P_1 & 320 & 30 & 10C_4 & G_{Q_8}(Q_8,C_4) & 64 & 14 & 2C_4, 4C_2, 4C_2 \\ 
	%\inpro{320, 1581}
K & 32 & 10 &  & G_{Q_8}(Q_8,C_2) & 32 & 10 & 2C_2,2C_2,2C_2,2C_2,2C_2 \\[0.1cm]
	%\inpro{32, 50}

%&&&&&&&\\[-0.3cm]
\hline
%\multicolumn{7}{l}{\hbox{\footnotesize The above fiducial vectors are for the reflection group with complex reflection group as
%a subgroup. To find the corresponding fiducial vectors for Cohen's presentation of the
%group multiply the vectors by $u$.  } }
\end{tabular}
\end{table}

\section{MUB symmetries}

The $P$ groups map the ten MUB lines of (\ref{fiveMUBs}) to themselves, i.e.,
are symmetries for them. This is easily seen by action of the
generators (\ref{P3nice}) for $P_3$ on the lines, e.g.,
$$ \pmat{1&1\cr1&-1}\pmat{1\cr0}=\pmat{1\cr1}, \quad
\pmat{1&1\cr1&-1}\pmat{0\cr1}=\pmat{1\cr-1}, \quad
\pmat{1&1\cr1&-1}\pmat{1\cr q} %=\pmat{1+q\cr1-q}
=\pmat{1\cr-q }(1+q). $$
Moreover, the columns of each matrix in $P_3$ are a MUB pair.
We expect that $P_3\subset U_2(\HH)$ is the (full) symmetry group
of the ten MUB lines.

We order the ten MUB lines as in (\ref{fiveMUBs}) with the $\pm$ entries 
ordered $+$, $-$. With this labelling, the generators in Table \ref{PKgenrefs}
correspond to the following permutations
$$ \pmat{i&0\cr0&1} \corr (3\, 6\, 4\, 5)(7\, 9\, 8\, 10), \qquad
{1\over2} \pmat{1+j&-1+j\cr-1+j&1+j} \corr (1\, 7\, 2\, 8)(5\, 10\, 6\, 9), $$ 
$$ \pmat{j&0\cr0&1} \corr (3\, 8\, 4\, 7)(5\, 10\, 6\, 9), \qquad
{1\over\sqrt{2}}\pmat{1&1\cr1&-1} \corr (1\, 3)(2\, 4)(5\, 6)(7\, 8)(9\, 10), $$
and the matrix $t$ of (\ref{Tmat}) to
$$ t={1\over 2}\pmat{j-k& 1-i\cr-j-k&1+i} \corr (1\, 6\, 9\, 4\, 7)(2\, 5\, 10\, 3\, 8). $$
The kernel of the action of $P_3$ on the MUB lines is $\inpro{-I}$, i.e.,
$P_3/\inpro{-I}$ acts faithfully on the ten lines. 
It is clear from the above permutations that $P_3$ acts on the five MUB pairs.
With these ordered as in (\ref{fiveMUBs}), the permutations corresponding to the
above elements are
$$ (2\, 3)(4\, 5), \qquad
(1\, 4)(3\, 5), \qquad
(2\, 4)(3\, 5), \qquad
(1\, 2),  $$
respectively. The kernel of this action on the five MUB pairs is the 
imprimitive reflection group $K$, and
$$ P_3/K \cong S_5, $$
i.e., any permutation of the five MUB pairs is possible. Similarly, we have
$$ P_2/K \cong A_5, \qquad P_1/K \cong D_5 \ \hbox{(dihedral group of order $10$)}. $$
The quotients of $P_j/K$ above are discussed in the proof of Theorem 4.2 in 
\cite{C80}, with the representation for $P_1$ dating back to Crowe \cite{Cr59}. 
The groups of type $P$ were introduced in Proposition 4.1 of \cite{C80} as
\begin{itemize}
\item $G$ (of rank $2$) is an extension of a subgroup of $S_6$ by ${\bf D}_2\circ D_4$
(where ${\bf D}_2\circ D_4\cong K$).
\end{itemize}
%This description is not inconsistent with the new primitive reflection
The imprimitive reflection group $P_0$ of Example \ref{newrefgroup}
can be considered to be of type P, as follows. 
%Denote it by 
%$$ P_0=\inpro{ \pmat{0&1\cr1&0}, \pmat{0&i\cr-i&0}, \pmat{0&j\cr-j&0}, 
%{1\over\sqrt{2}}\pmat{1&1\cr1&-1} }. $$
Its generators (the first three are in $K$) permute the MUB pairs as follows
$$ \pmat{0&b\cr b^{-1}&0},\ b\in\{1,i,j\} \corr (), \qquad
{1\over\sqrt{2}}\pmat{1&1\cr1&-1} \corr (1\, 2), $$
so that 
$$ P_0/K \cong \inpro{(1\, 2)} = C_2, $$
i.e., $P_0$ is of type $P$.
The group $P_0$ is not normal in $P_3$, having five conjugates.

%$$
%\pmat{0&1&0&0 \cr 1&0&0&0\cr 0&0&0&1\cr 0&0&1&0}, \
%\pmat{0&i&0&0 \cr -i&0&0&0 \cr 0&0&0&-i \cr 0&0&i&0}, \
%\pmat{0&0&0&-1 \cr 0&0&1&0 \cr 0&1&0&0 \cr -1&0&0&0}, \
%{1\over\sqrt{2}}\pmat{1&1&0&0\cr 1&-1&0&0\cr 0&0&1&1\cr 0&0&1&-1} $$

\section{Spherical designs and small sets of invariant lines}

We have seen that the orbit of the vector/line $v=e_1$ (or any MUB line) 
under the action of $P_3$ (and its subgroups $P_1$, $P_2$) is the
MUB lines of (\ref{fiveMUBs}).
These ten lines are the roots of the reflection groups $P_1$ and $P_2$,
but not $P_3$ (which has $40$ additional root lines).

These lines are well-spaced, in the sense that the set of angles $|\inpro{v,w}|^2$ between
different lines given by unit vectors $v,w\in\HH^2$ is the small set $\{0,{1\over2}\}$.
We will give a related notion of being well-spaced, that of
being a ``spherical design'' \cite{DGS77}, which corresponds to the lines being a cubature rule for the
sphere. In this section, we will use a general method, which does not require the
groups $G\subset U_d(\HH)$ involved be reflection groups, 
to find small sets of $G$-invariant lines which provide 
good spherical designs.
This will give the ten MUB lines, and also other interesting configurations
(see Table \ref{Fiducialtable}).

Let $t\in\{1,2,\ldots\}$. 
The set of lines given by $n$ unit vectors $(v_j)$ in $\Hd$ is called a
{\bf spherical $(t,t)$-design} for $\Hd$ if they give equality in the 
inequality
\begin{equation}
\label{t-designdef}
\sum_{j=1}^n\sum_{k=1}^n |\inpro{v_j,v_k}|^{2t} 
\ge c_t(\Hd) \Bigl(\sum_{\ell=1}^n \norm{v_\ell}^{2t}\Bigr)^2,
\qquad c_t(\Hd):=\prod_{j=0}^{t-1} {2+j\over 2d+j},
\end{equation}
i.e., 
\begin{equation}
\label{t-designdef}
\sum_{j=1}^n\sum_{k=1}^n |\inpro{v_j,v_k}|^{2t}
= c_t(\Hd) \Bigl(\sum_{\ell=1}^n \norm{v_\ell}^{2t}\Bigr)^2,
%\qquad c_t(\Hd):=\prod_{j=0}^{t-1} {2+j\over 2d+j},
\end{equation}
see \cite{W20c} for details.
These can be viewed as a cubature rule for the unit sphere
in $\Hd$, and are equivalent to the quaternionic {\bf spherical $t$-designs} of
\cite{H82}. Therefore, we are interested in large values of $t$ and small numbers
of lines $n$. We note
$$ 
c_1(\Hd) = {1\over d}, \qquad 
c_2(\Hd) = {3\over d(2d+1)}, \qquad 
c_3(\Hd) = {6\over d(2d+1)(d+1)}. $$

It follows from (\ref{t-designdef}) that the orbit $(g x)_{g\in G}$ of a
nonzero vector $x\in\Hd$ under the action 
of finite group of unitary matrices $G\subset U_d(\HH)$
is a spherical $(t,t)$-design if and only if
\begin{equation}
\label{pG(t)defn}
p_G^{(t)}(x) := {1\over|G|}\sum_{g\in G}|\inpro{x,gx}|^{2t}-c_t(\Hd) \inpro{x,x}^{2t} = 0.
\end{equation}
Moreover (see \cite{W20c}), $G$ is irreducible if and only if 
every orbit is a $(1,1)$-design, i.e.,
\begin{equation}
\label{irreduciblecdn}
{1\over|G|}\sum_{g\in G}|\inpro{x,gx}|^{2}-{1\over d}\inpro{x,x}^{2} = 0.
\end{equation}

By direct calculation of (\ref{pG(t)defn}) in {\tt magma}, we obtain the following.

\begin{proposition}
Every orbit of a $P$ group, i.e., $P_1$, $P_2$, $P_3$,
		is a spherical $(3,3)$-design.
\end{proposition}

\begin{proof} We have $p_G^{(3)}=0$, for $G=P_1,P_2,P_3$.
\end{proof} 

In particular, the ten MUB lines are a spherical $(3,3)$-design for 
$\HH^2$. This can be verified directly by evaluating $p_G^{(3)}$ at
a unit MUB vector, which gives
$$ {1\over 10}\Bigl( 1\cdot 1^3+1\cdot 0^3+8\cdot\bigl({1\over2}\bigr)^3\Bigr)
-{6\over 2\cdot 5\cdot3}\cdot 1= {1\over5}-{1\over5}=0. $$
We observe that
\begin{itemize}
\item The ten MUB lines are a spherical $(3,3)$-design by being
	the orbit of a $P$ group.
\item There is a small number of these vectors, as their stabiliser subgroups
are large.
\end{itemize}
Since the stabiliser group of a line is, by definition, reducible, we can 
find small sets of lines such as the MUB lines as follows:
\begin{itemize}
\item Find the large reducible subgroups of the $P$ groups,
	and the lines they stabilise.
%\item The orbit of the stabilised line is  therefore a $(3,3)$-design with a small number of vectors.
\item The orbit of the stabilised line is then a $(3,3)$-design with
a small number of vectors.
\end{itemize}
A line stabilised by a proper subgroup is called a {\bf fiducial vector} (or line).

The condition (\ref{irreduciblecdn}) allows us to identify the reducible subgroups of
the $P$ groups. These need not be reflection groups. The only other technical 
condition is determining those $v\in\Hd$ 
which give a line stabilised by some $g\in U_d(\HH)$, 
i.e.,
\begin{equation}
\label{eigcdn}	
gv = v\gl, \qquad \hbox{for some $\gl\in\HH$}.
\end{equation}
Nominally, $v$ appears to be an eigenvector for $g$, but the calculation
$$ g(v\ga)=v\gl\ga=v\ga (\ga^{-1}\gl\ga), \qquad \ga\in\HH, $$
shows that there is no natural associated eigenvalue when $\gl$ is not real
(we will still refer to $v$ as eigenvector).
Nevertheless, the condition (\ref{eigcdn}) can be verified (without calculating a $\gl$),
as the condition which gives equality in the Cauchy-Schwartz inequality, i.e.,
\begin{equation}
\label{fixedvectoreqns}
|\inpro{v,gv}|^2 = \inpro{v,v}^2, \qquad\forall g\in G,
\end{equation}
which gives a quartic polynomial in the $1$, $i$, $j$, $k$ 
parts of the coordinates of $v\in\Hd$.

%We now outline the computations undertaken in {\tt magma}, as detailed above, for $G=P_1$. 
We now outline our computations in {\tt magma}, as detailed above, for $G=P_1$. 
\begin{itemize}
\item We will take the ``large'' reducible subgroups of $G$, to be 
those which are maximal.
\item The corresponding systems of lines will be of minimal size.
\item The lines for the proper subgroups of the maximal reducible subgroups 
	(which are automatically reducible) will either be lines 
		for the maximal reducible subgroup, or a larger set of lines. 
\end{itemize}

We observe that a vector $v^\perp$ orthogonal to the line given by a $v\in\HH^2$ with
real first component is given by the formula
\begin{equation}
\label{vperpform}
v = \pmat{a\cr b}, \quad a\in\RR, \qquad v^\perp=\pmat{-\overline{b}\cr a},
\end{equation}
e.g., the roots of (\ref{40rootlinesP3}) appear as orthogonal pairs, 
and the MUB pairs can be written
$$ \pmat{1\cr0},\pmat{0\cr1}, \quad
\pmat{1\cr1},\pmat{-1\cr1}, \quad
\pmat{1\cr i},\pmat{i\cr1}, \quad
\pmat{1\cr j},\pmat{j\cr1}, \quad
\pmat{1\cr k},\pmat{k\cr1}. $$

The subgroups of $G\subset M_d(\HH)$ can be computed in {\tt magma} using the command
{\tt Subgroups(G)}, and the lattice by {\tt SubgroupLattice(G)}, the latter 
only working for complex (symplectic) presentations of the group. 
For the group $P_1$ of order $320$, the reducible subgroups and the lengths
of their conjugacy classes and number of lines are
\vskip0.3truecm
%\begin{table}[!h]
	\begin{tabular}{ p{6truecm} p{0truecm} p{6truecm} }
 % \caption{Reducible subgroups }
%\label{TOIgroups-table}
\begin{tabular}{ |  >{$}l<{$} | >{$}l<{$} | >{$}l<{$} | >{$}l<{$} |}
\hline
&&&\\[-0.4cm]
\hbox{order} & \hbox{number} & \hbox{length} & \hbox{lines} \\[0.1cm]
\hline
&&&\\[-0.4cm]
32 & 1 & 5 & 10  \\
20 & 1 & 16 & 16 \\
16 & 3 & 5 & 20 \\
   & 4 & 10 & 20 \\
10 & 1 & 16 & 32 \\
8 & 5 & 5 & 40 \\
  & 6 & 10 & 40 \\
\hline
%\multicolumn{7}{l}{\hbox{\footnotesize The above fiducial vectors are for the reflection group with complex reflection group as
%a subgroup. To find the corresponding fiducial vectors for Cohen's presentation of the
%group multiply the vectors by $u$.  } }
\end{tabular}
&&
\begin{tabular}{ |  >{$}l<{$} | >{$}l<{$} | >{$}l<{$} | >{$}l<{$} |}
\hline
&&&\\[-0.4cm]
\hbox{order} & \hbox{number} & \hbox{length} & \hbox{lines} \\[0.1cm]
\hline
&&&\\[-0.4cm]
5 & 1 & 16 & 64 \\
4 & 3 & 5 & 80 \\
  & 2 & 10 & 80 \\
  & 1 & 20 & 80 \\
2 & 1 & 1 & 160 \\
2 & 1 & 10 & 160 \\
1 & 1 & 1 & 320 \\
%&&&&&&&\\[-0.3cm]
\hline
%\multicolumn{7}{l}{\hbox{\footnotesize The above fiducial vectors are for the reflection group with complex reflection group as
%a subgroup. To find the corresponding fiducial vectors for Cohen's presentation of the
%group multiply the vectors by $u$.  } }
\end{tabular}
\end{tabular}
%\end{table}
\vskip0.3truecm
Of these, there are four are maximal irreducible subgroups, i.e.,
%$$ H=\inpro{ \pmat{i&0\cr0&1}, \pmat{i&0\cr0&i}, \pmat{j&0\cr0&j} }=Stab(P_1,e_1),
%\quad \hbox{(Order $32$, $10$ lines)}, $$
%$$ H=\inpro{ \pmat{j&0\cr0&k},{1\over2}\pmat{i-j&i-j\cr i-j&-i+j}},
%\quad \hbox{(Order $20$, $16$ lines)},$$
%$$ H=\inpro{\pmat{j&0\cr0&k},\pmat{0&j\cr j&0}}, \quad \pmat{\sqrt{2}\cr1+i},
%\quad \hbox{(Order $16$, $20$ lines)}, $$
%$$ H=\inpro{\pmat{j&0\cr0&j},{1\over2}\pmat{i-k&1-j\cr1+j&-i-k}}, \quad \pmat{\sqrt{2}\cr1+j},
%\quad \hbox{(Order $16$, $20$ lines)}, $$
\begin{align}
& H=\inpro{ \pmat{i&0\cr0&1}, \pmat{i&0\cr0&i}, \pmat{j&0\cr0&j} }, 
	&\quad \hbox{(Order $32$, $10$ lines)}, \label{H32} \\
& H=\inpro{ \pmat{j&0\cr0&k},{1\over2}\pmat{i-j&i-j\cr i-j&-i+j}},
	&\quad \hbox{(Order $20$, $16$ lines)}, \label{H20} \\
& H=\inpro{\pmat{j&0\cr0&k},\pmat{0&j\cr j&0}},  %\quad \pmat{\sqrt{2}\cr1+i},
	&\quad \hbox{(Order $16$, $20$ lines)}, \label{H16a} \\
& H=\inpro{\pmat{j&0\cr0&j},{1\over2}\pmat{i-k&1-j\cr1+j&-i-k}}, % \quad \pmat{\sqrt{2}\cr1+j},
	&\quad \hbox{(Order $16$, $20$ lines)}. \label{H16b} 
\end{align}
None of these are reflection groups.
The first is diagonal, and it is easy to see that it stabilises 
the lines given by standard basis vectors, and no others. 
The five conjugates of this group in $P_1$ stabilise the five MUB pairs, 
and no other lines. Thus the $P_1$-orbit of any line fixed by the reducible subgroup 
of order $32$ is the ten MUB lines, i.e., we arrive at the MUB lines without
using the fact that $P_1$ is a reflection group.

We now consider the irreducible subgroup of order $20$, which gives $16$ lines.
%giving $16$ lines (with $16$ conjugates). One (nice) stabilised vector is
%$$ \pmat{1+\sqrt{5}\cr1+i+j+k}, $$
%with stabiliser group
Since the line given by $e_2$ is not fixed, we can suppose that a fiducial 
(stabilised) vector has the form
\begin{equation}
\label{vpolyform}
v= \pmat{1\cr x_1+x_2 i+x_3 j+x_4 k}, \quad x_1,x_2,x_3,x_4\in\RR,
\end{equation}
so that (\ref{vperpform}) gives 
$$ v^\perp= \pmat{-x_1+x_2 i+x_3 j+x_4 k\cr1}. $$
Taking $g$ in (\ref{fixedvectoreqns}) to be the two generators of (\ref{H20}), gives 
the two quartic equations
%\begin{align*}
%& (x_1-x_2)^2+(x_3-x_4)^2=0, \Implies x_2=x_1, \ x_4=x_3, \cr
%& ( x_1^4 + x_2^4 + x_3^4 + x_4^4 + 2x_1^2 x_2^2 + 2x_1^2 x_3^2 + 2 x_1^2 x_4^2 + 2x_2^2x_3^2 + 2x_2^2x_4^2 + 2x_3^2x_4^2 ) \cr
%	& \qquad + (4x_1^3 + 4x_1x_2^2 + 4x_1x_3^2 + 4x_1x_4^2 )
%+( 2x_1^2 + 2x_2^2 + 2x_3^2 -2x_4^2 -8x_2x_3 -4x_1 ) + 1 =0.
%\end{align*}
$$ (x_1-x_2)^2+(x_3-x_4)^2=0, \Implies x_2=x_1, \ x_4=x_3, $$
\begin{align*}
& ( x_1^4 + x_2^4 + x_3^4 + x_4^4 + 2x_1^2 x_2^2 + 2x_1^2 x_3^2 + 2 x_1^2 x_4^2 + 2x_2^2x_3^2 + 2x_2^2x_4^2 + 2x_3^2x_4^2 ) \cr
        & \qquad + (4x_1^3 + 4x_1x_2^2 + 4x_1x_3^2 + 4x_1x_4^2 )
+( 2x_1^2 + 2x_2^2 + 2x_3^2 -2x_4^2 -8x_2x_3 -4x_1 ) + 1 =0.
\end{align*}
The Gr\"obner for these equations provided by {\tt GroebnerBasis(I)} in {\tt magma} is
involved, and we were unable to automate the calculation of fiducials. 
The set of equations (\ref{fixedvectoreqns}) given by all elements of $H$ of (\ref{H20}), 
which are not $\pm I$, 
consists of the first equation and six of a similar complexity to the second.
The Gr\"obner basis provided by {\tt magma} for the ideal given 
by taking these seven equations is nicer, with the following equation for $x_4$
$$    (4 x^2  + 2 x - 1)^3 =0 \Implies x_4={-1\pm \sqrt{5}\over 4}. $$
%which leads to a single Fiducial (given by $x={\sqrt{5}-1\over4}$).
This leads to two fiducials, which are orthogonal, i.e.,
\begin{equation}
\label{w-wperpdefn}
w=\pmat{1+\sqrt{5}\cr1+i+j+k}, \qquad w^\perp %=\pmat{\sqrt{5}-1\cr-1-i-j-k}
=\pmat{-1+i+j+k\cr1+\sqrt{5}}.
\end{equation}
The orbits of $w$ and $w^\perp$ under $P_1$ and $P_2$ are $16$ lines
with angles $\{{1\over5},{3\over5}\}$, given by 
\begin{align}
	w: & \qquad \pmat{1+\sqrt{5}\cr q(1+i+j+k)}, \
 \pmat{q(1+i+j+k)\cr 1+\sqrt{5}},  \quad q\in Q_8, \label{w-orbit}\\
 w^\perp: &\qquad \pmat{1+\sqrt{5}\cr q(-1+i+j+k)}, \
 \pmat{q(-1+i+j+k)\cr 1+\sqrt{5}},  \quad q\in Q_8, \label{wp-orbit}
\end{align}
and the orbit of $w$ or $w^\perp$ under $P_3$ gives 
$32$ lines, which are the union of the two sets, 
having angles $\{0,{1\over5},{2\over5},{3\over5},{4\over5}\}$. 
%Let $\cL$ be the set of $16$ lines given by the $P_1$-orbit of $w$ or $w^\perp$.
%With $H$ given by (\ref{H20}), we have
%$$ \pmat{j&0\cr0&k} \in H=\Stab(P_1,w)=\Stab(P_1,w^\perp), $$
%so that 
%$$ \pmat{1+\sqrt{5}\cr\ga}\in\cL \Implies 
%\pmat{j(1+\sqrt{5})\cr k\ga}\in\cL \Implies
%\pmat{(1+\sqrt{5})\cr -k\ga j}\in\cL \Implies $$

The spherical $3$-design of $16$ lines constructed above meets the special
bound of \cite{H79}, \cite{H82} for the number of lines in $\HH^2$ with exactly two nonzero angles.

\begin{example} (Special bound) Hoggar \cite{H79} provides two bounds on the 
number $n$ of vector/lines in $\Hd$ with a finite angle set (indeed having 
a finite number of angles implies a set of lines is finite). If there are two 
nonzero angles $A=\{\ga,\gb\}$, then there is a special bound (depending on the 
angles), and an absolute bound (not depending on the angles), 
given by
%which for $\HH^2$ are 
%$$ n\le{10(1-\ga)(1-\gb)\over3-5(\ga+\gb)+10\ga\gb} , \qquad n\le20. $$
$$ n\le{d(2d+1)(1-\ga)(1-\gb)\over3-(2d+1)(\ga+\gb)+d(2d+1)\ga\gb}, \qquad
	n\le {1\over3} d^2(4d^2-1). $$
For $A=\{{1\over5},{3\over5}\}$ ($d=2$), the special bound gives $n\le 16$,
i.e., $3$-design of $16$ lines in $\HH^2$ given by (\ref{w-orbit}) or (\ref{wp-orbit}) 
meets the special bound. The only other known cases where this %e two nonzero angle 
special bound is met is for a $2$-design of $15$ lines in $\HH^2$ with angles 
	$\{{1\over4},{5\over8}\}$ obtained from a reflection group of type O 
	\cite{W24}, and for a  $2$-design of $64$ lines in $\HH^4$ with angles $\{{1\over9},{1\over3}\}$ obtained
	from a quaternionic polytope (Example 22, \cite{H82}).

The examples from the O and P groups disprove the following conjecture of \cite{H82}
(for the special bound):

{\smallskip\narrower\noindent
\em
Conjecture 1: Whenever a special or absolute bound is attained, each nonzero $\ga\in A$
has the form $1/p$ ($p\in\NN$), or is irrational.\par\smallskip
}

\noindent
This could still be true for the roots of a reflection group, which were the bulk of
cases considered in \cite{H82}.
\end{example}

Similar calculations for the maximal reducible subgroups $H$ of $P_1$, with order $16$, given 
by (\ref{H16a}) and (\ref{H16b}), yield a single fiducial vector for each, i.e.,
$$ \pmat{\sqrt{2}\cr1+i}, \qquad \pmat{\sqrt{2}\cr1+j}, $$
with the corresponding $P_1$-orbits of $20$ lines, 
with angles $\{0,{1\over4},{1\over2},{3\over4}\}$,
given by 
\begin{align}
\pmat{\sqrt{2}\cr i^m\ga},\ \pmat{1+\sqrt{2}\cr i^m j}, \ \pmat{i^m j\cr1+\sqrt{2}}, 
	& \qquad \ga\in\{1+i,i+j,i-j\},\quad m=0,1,2,3, \label{H16a-orbit} \\
\pmat{\sqrt{2}\cr i^m\ga},\ \pmat{1+\sqrt{2}\cr i^m }, \ \pmat{i^m \cr1+\sqrt{2}}, 
	& \qquad \ga\in\{1+j,1-j,j-k\},\quad m=0,1,2,3. \label{H16b-orbit}
\end{align}
We observe that (\ref{H16a-orbit}), (\ref{H16b-orbit}) is a partition of the $40$
root lines of $P_3$ given by (\ref{40rootlinesP3}). 

If $s$ is the number of angles in a spherical $t$-design, and $t\ge s-1$, then 
it is a {\em regular scheme} (see \cite{H84}). 
Hence, the $3$-designs of $10$, $16$, $20$ vectors/lines that we have constructed are 
regular schemes, since each satisfies $s\le 4$.

\begin{example} 
\label{P2reduce}
The maximal reducible subgroups of $P_2$ have orders $192$, $120$, $48$, $24$, 
which correspond to systems of $10$, $16$, $40$, $80$ lines.
The first two of these groups have (\ref{H32}) and (\ref{H20}) as subgroups, respectively,
and so give the sets $10$ and $16$ lines obtained from $P_1$. The third group has
(\ref{H16a}) and (\ref{H16b}) as subgroups, 
and so fixes the line of both $w$ and $w^\perp$ of (\ref{w-wperpdefn}),
with either of them being a fiducial for the set of $40$ lines consisting of the union
of their $P_1$-orbits (\ref{w-orbit}) and (\ref{wp-orbit}).

The irreducible subgroup of order $24$, which gives $80$ lines, is
\begin{equation}
\label{H80}
H=\inpro{ {1\over2}\pmat{0&1-i-j+k\cr 1+i+j+k&0}, 
{1\over2}\pmat{i+j&i-j\cr i-j&i+j} }.
\end{equation}
This fixes the lines given by the orthogonal vectors
\begin{equation}
\label{80fiducials}
\pmat{\sqrt{3}\cr1+i+j}, \qquad \pmat{-1+i+j\cr\sqrt{3}}, % \pmat{\sqrt{3}\cr-1-i-j}, 
\end{equation}
and the $P_2$-orbit of either of these fiducials is the same set of $80$ lines 
at angles $\{0,{1\over6},{2\over6},{3\over6},{4\over6},{5\over6}\}$.
\end{example}

The $P_1$-orbits of the fiducials of (\ref{80fiducials}) give $40$ lines 
(a partition of the $80$) 
with the same angles, i.e., $\{0,{1\over6},{2\over6},{3\over6},{4\over6},{5\over6}\}$. 
Since this set of $40$ lines has not yet appeared, its stabiliser in $P_1$, 
which has order $8$, must not be a maximal reducible subgroup of $P_1$.
This implies that it fixes a line in one of the sets of $10$, $16$, $20$ lines
obtained from the maximal reducible subgroups of $P_1$. Our direct verification of this
fact below, leads to an intriguing example of a ``continuous family'' of eigenvectors
for a matrix group (over the quaternions).

\begin{example}
\label{P1nonmaximal}
Let $G$ be the stabiliser in $P_1$ of the lines given by the fiducial
vectors of (\ref{80fiducials}), i.e.,
$$ G := H\cap P_1 = \inpro{\pmat{0&k\cr k&0},{1\over2}\pmat{i+k&1+j\cr -1+j&i-k}}, 
	\qquad |G|=8,$$
where $H$ is given by (\ref{H80}). The Gr\"obner basis for the equations (\ref{fixedvectoreqns})
for a fixed line given by a $v\in\HH^2$ of the form (\ref{vpolyform}) include
$x_4^3=0$ and $(x_1-x_3)^2=0$, and reduce to
$$ x_3=x_1, \qquad x_4=0, \qquad 2x_1^2+x_2^2-1=0. $$
Given that $x_2^2=1-2x_1^2$, we may solve these equations to get
$$ x_1=x_3=t, \quad t^2\le{1\over2}, \qquad x_2=\pm\sqrt{1-2t^2}, \qquad x_4=0, $$
and hence obtain a continuous family of fiducials (eigenvectors)
$$ v=\pmat{1\cr t(1+j)\pm\sqrt{1-2t^2} i}, \quad %-{1\over\sqrt{2}}\le t\le {1\over\sqrt{2}}.
\hbox{$t\in[-{1\over\sqrt{2}},{1\over\sqrt{2}}]$}. $$
We observe that for the special cases $t=0,{1\over\sqrt{2}},{1\over\sqrt{3}}$, we obtain
fiducial vectors for $P_1$ giving $10$, $16$, $40$ lines, i.e.,
$$ 
t=0:\quad \pmat{1\cr\pm i}, \qquad
\hbox{$t={1\over\sqrt{2}}$}:\quad \pmat{1\cr{1\over\sqrt{2}}(1+j)}, \qquad
\hbox{$t={1\over\sqrt{3}}$}:\quad \pmat{1\cr{1\over\sqrt{3}}(1\pm i+j)}.  $$
Thus $G$ is a proper subgroup of the maximal reducible subgroups of $P_1$ given by 
$$ \Stab(P_1,\pmat{1\cr i}), \qquad \Stab(P_1,\pmat{\sqrt{2}\cr1+j}). $$
\end{example}

\begin{example} 
\label{80linesP3}
Of the $106$ irreducible subgroups of $P_3$, five are maximal,
and one of these 
$$ H = \inpro{{1\over\sqrt{2}}\pmat{i+j &0\cr0&i+j},
{1\over2}\pmat{1+i&1+i\cr-1+i&1-i} }\qquad |H|=48, $$
fixes the orthogonol vectors
$$ \pmat{3\cr 1+i+j}, \qquad \pmat{-1+i+j\cr3}, $$
each of which is a fiducial vector for a set of $80$ lines with angles
$\{0,{1\over8},{2\over8},\ldots,{7\over8}\}$. 
%has five maximal irreducible subgroups, of which one gives a new set of lines. Take as the fiducial

	The other four maximal reducible subgroups give sets of $10$, $32$, $40$, $80$ lines
	already obtained (they have a larger symmetry group $P_3$). 
\end{example}

\begin{table}[!h]
\caption{The line systems given by the maximal reducible subgroups of the $P$ groups } 
\label{Fiducialtable}
\begin{center}
\begin{tabular}{ | >{$}l<{$} | >{$}l<{$} | >{$}l<{$} | >{$}l<{$} | >{$}l<{$} | l | }
%\multicolumn{6}{l}{\hbox{\bf The line systems given by the maximal reducible subgroups of the $P$ groups } } \\[0.3cm]
\hline
	&&&&&\\[-0.3cm]
	G & |H| & \hbox{fiducial} & \hbox{angles} & \hbox{lines} & comment \\[0.1cm]
\hline
	&&&&&\\[-0.3cm]
	P_1^\dagger, P_2^\dagger, P_3^\dagger & 32, 192, 384 & \pmat{1\cr 0} & 0, {1\over2} & 10 & 
	(\ref{fiveMUBs}) \ MUBs 
	\\[0.4cm]
P_1^\dagger, P_2^\dagger & 20, 120 & \pmat{1+\sqrt{5}\cr 1+i+j+k} & {1\over5},{3\over5}& 16 & (\ref{w-orbit}) \\[0.4cm]
& & \pmat{ -1+i+j+k \cr 1+\sqrt{5} } & {1\over5},{3\over5}& 16 & (\ref{wp-orbit})  \\[0.4cm]
P_3^\dagger & 120 & \pmat{1+\sqrt{5}\cr 1+i+j+k} & 0,{1\over5},{2\over5},{3\over5},{4\over5}& 32 & (\ref{w-orbit}), (\ref{wp-orbit}) \\[0.4cm]
	P_1^\dagger & 16 & \pmat{\sqrt{2}\cr 1+i} & 0,{1\over4},{1\over2},{3\over4} & 20 & (\ref{H16a-orbit})  \\[0.4cm]
	P_1^\dagger & 16 & \pmat{\sqrt{2}\cr 1+j} & 0,{1\over4},{1\over2},{3\over4} & 20 & (\ref{H16b-orbit}) \\[0.4cm]
	P_2^\dagger, P_3^\dagger & 48, 96 & \pmat{\sqrt{2}\cr 1+i} & 0,{1\over4},{1\over2},{3\over4} & 40 & (\ref{H16a-orbit}), (\ref{H16b-orbit})   \\[0.4cm]
	 P_1 & 8 & \pmat{\sqrt{3}\cr 1+i+j} & 0,{1\over6},{1\over3},{1\over2},{2\over3},{5\over6} & 40 & Example \ref{P1nonmaximal} \\[0.4cm]
	 P_2^\dagger, P_3^\dagger & 24, 48 & \pmat{\sqrt{3}\cr 1+i+j} & 0,{1\over6},{1\over3},{1\over2},{2\over3},{5\over6} & 80 & Example \ref{P2reduce} \\[0.4cm]
	P_1, P_2, P_3^\dagger & 4, 24, 48 & \pmat{3\cr 1+i+j} & 0,{1\over8},{2\over8},\ldots,{7\over8} & 80 & Example \ref{80linesP3} \\[0.4cm]
\hline
\multicolumn{6}{l}{\hbox{\footnotesize $\dagger$ The stabiliser $H$ of the fiducial 
vector in $P_j$ is a maximal reducible subgroup of $G=P_j$.  } }
\end{tabular}
\end{center}
\end{table}

The behaviour uncovered in the Example \ref{P1nonmaximal}, i.e.,
%that a nonscalar matrix of size $2$ over the quaternions can have
that a nonscalar $2\times2$ matrix over the quaternions can have
%that a nonscalar matrix in $M_2(\HH)$s can have
a continuous family of right eigenvectors, 
is completely different from the complex case (the eigenvalues are uniquely defined and there are at most
two eigenvector lines in $\CC^2$), and hence of some interest (see \cite{Z97}, \cite{FWZ11}).
We now give a variant which illustrates %of this example which shows % illustrates 
some of the mechanics of this phenomenon.

%$P_3$ has $106=211-205$ irreducible subgroups.

%$$ \pmat{1+\sqrt{5}\cr1+i+j+k}, \
%\pmat{1-i-j+k\cr1+\sqrt{5}}, \ 
%\pmat{1+i+j+k\cr1+\sqrt{5}}, \ 
%\pmat{1+\sqrt{5}\cr1-i-j+k}, 
%$$
%now multiplying coordinates by appropriate powers of $i$ gives
%$$ 
%\pmat{1+\sqrt{5}\cr-1+i-j+k}, \
%\pmat{1+i-j-k\cr1+\sqrt{5}}, \ 
%\pmat{-1+i-j+k\cr1+\sqrt{5}}, \ 
%\pmat{1+\sqrt{5}\cr1+i-j-k}, 
%$$
%$$
%\pmat{1+\sqrt{5}\cr-1-i-j-k}, \
%\pmat{-1+i+j-k\cr1+\sqrt{5}}, \
%\pmat{-1-i-j-k\cr1+\sqrt{5}}, \
%\pmat{1+\sqrt{5}\cr-1+i+j-k},
%$$
%$$ 
%\pmat{1+\sqrt{5}\cr1-i+j-k}, \
%\pmat{-1-i+j+k\cr1+\sqrt{5}}, \ 
%\pmat{1-i+j-k\cr1+\sqrt{5}}, \ 
%\pmat{1+\sqrt{5}\cr-1-i+j+k}, 
%$$

\begin{example} 
The reducible subgroup of order $8$ given by 
	$$ H=\inpro{ \pmat{i&0\cr0&i}, \pmat{0&j\cr j&0}} , $$
has a continuous family of eigenvectors given by
	$$ \pmat{1\cr t\pm\sqrt{1-t^2}i}, \qquad -1\le t\le 1. $$
The corresponding eigenvalues can be determined by verifying this, e.g.,
\begin{align*} & \pmat{0&j\cr j&0} \pmat{1\cr t\pm\sqrt{1-t^2}i}
= \pmat{tj\pm\sqrt{1-t^2}ji\cr j}
= \pmat{tj\mp\sqrt{1-t^2}k\cr j} \cr
&\qquad = \pmat{1 \cr j (-tj\pm\sqrt{1-t^2}k) } (tj\mp\sqrt{1-t^2}k)
= \pmat{1 \cr t\pm\sqrt{1-t^2} i } (tj\mp\sqrt{1-t^2}k).
\end{align*} 
Thus the second matrix can be diagonalised in multiple ways, e.g.,
$$ M^{-1} \pmat{0&j\cr j&0} M = \pmat{j&0\cr 0&k}, \pmat{j&0\cr 0&{j+k\over\sqrt{2}}},
\qquad M=\pmat{1&1\cr1&-i},\pmat{1&1\cr1&{1-i\over\sqrt{2}}}. $$
%$$ M_1^{-1} \pmat{0&j\cr j&0} M_1 = \pmat{j&0\cr 0&k}, \quad
%	M_1=\pmat{1&1\cr1&-i}, \qquad $$
%	$$ M_2^{-1} \pmat{0&j\cr j&0} M_2 = \pmat{j&0\cr 0&{j+k\over\sqrt{2}}}, \quad
%	M_2=\pmat{1&1\cr1&{1-i\over\sqrt{2}}}, \qquad $$
\end{example}

Related to the projective stabiliser group of a line, is the (pointwise) stabiliser group 
of a vector (or set of vectors), which is reducible. 
It was shown in \cite{BST23}, \cite{S23} that for a quaternionic 
reflection group, these so called {\it parabolic} subgroups are reflection groups
(this is a classical result of Steinberg for complex reflection groups).
For the fiducial vectors of Table \ref{Fiducialtable}, we calculated the pointwise stabiliser group
(a subgroup of the projective stabiliser). These were all trivial, except for the cases
%$$ v=\pmat{1\cr0}:\quad \pmat{1&0\cr 0&H}, \quad H=\inpro{i}\ (P_1), Q_8\ (P_1,P_2),
%\qquad \pmat{\sqrt{2}\cr1+i}: \quad \inpro{\pmat{0&{1-i\over\sqrt{2}}\cr{1+i\over\sqrt{2}}&0}}. $$
$$ \pmat{1\cr0}^{P_1} = \pmat{1&0\cr 0&\inpro{i}}, \quad
\pmat{1\cr0}^{P_2}, \pmat{1\cr0}^{P_3} = \pmat{1&0\cr 0&Q_8}, \qquad
\pmat{\sqrt{2}\cr1+i}^{P_3}=\inpro{\pmat{0&{1-i\over\sqrt{2}}\cr{1+i\over\sqrt{2}}&0}}. $$

%Take the pointwise stabilser of $e_1$ (this should be a reflection group), and 
%compare with the projective stabiliser.

\section{Concluding remarks}

The spherical $3$-designs with a small number of lines that we obtained as orbits of 
the $P$ groups are summarised in Table \ref{Fiducialtable}. 
The construction used is essentially that of ``highly symmetric
tight frames'' (see \cite{BW13}, \cite{IJM20}, \cite{G22}). The key idea is to go from a
quaternionic representation of an abstract group to finitely many
associated nice sets of lines.
In the spirit of Hoggar, we offer a conjecture informed by our calculations 
(see Example \ref{P1nonmaximal} and Table \ref{Fiducialtable}).

\begin{conjecture}
For every quaternionic reflection group (or every finite irreducible group of $d\times d$ matrices over $\HH$), the maximal reducible subgroups fix a finite number of lines.
\end{conjecture}

If this holds, then for a given group,
taking those lines given by the maximal reducible
subgroups gives a {\em finite} class of ``highly symmetric
tight frames for $\Hd$''.

In Table \ref{PKgenrefs}, we gave nice generators for the $P$ groups (none seem to 
appear in the literature, see \cite{H82} and \cite{C90}). 
Here is sample {\tt magma} code for their construction, and the fiducial $w$ of (\ref{w-wperpdefn}) which
gives $16$ lines.

%These can be constructed in {\tt magma} as follows:

\begin{verbatim}
F:=CyclotomicField(120); 
PR<t>:=PolynomialRing(F); 
rt2:=Roots(t^2-2)[1][1]; rt3:=Roots(t^2-3)[1][1]; rt5:=Roots(t^2-5)[1][1]; 

Q<i,j,k>:=QuaternionAlgebra<F|-1,-1>;

a:=Matrix(Q,2,2,[i,0,0,1]); b:=Matrix(Q,2,2,[j,0,0,1]);
c:=1/2*Matrix(Q,2,2,[1+j,-1+j,-1+j,1+j]); 
d:=1/rt2*Matrix(Q,2,2,[1,1,1,-1]);

P1:=MatrixGroup<2,Q|a,c>; 
P2:=MatrixGroup<2,Q|a,b,c>; 
P3:=MatrixGroup<2,Q|a,b,d>;

w:=Matrix(Q,2,1,[1+rt5,1+i+j+k]);
\end{verbatim}

Here is code for the Hermitian transpose (and hence the inner product
and angles).

\begin{verbatim}
// This gives the 1,i,j,k parts of a matrix or polynomial over Q 
HtoRparts := function(q);
  x1:=1/4*(q-i*q*i-j*q*j-k*q*k); x2:=1/4/i*(q-i*q*i+j*q*j+k*q*k);
  x3:=1/4/j*(q+i*q*i-j*q*j+k*q*k); x4:=1/4/k*(q+i*q*i+j*q*j-k*q*k);
  return [x1,x2,x3,x4];
end function;

HermTranspose := function(A);
  c:=HtoRparts(A);
  B:=c[1]-c[2]*i-c[3]*j-c[4]*k;
  return Transpose(B);
end function;
\end{verbatim}

\bibliographystyle{alpha}
\bibliography{references}

\begin{thebibliography}{ACFW18}

\bibitem[ACFW18]{ACFW18}
Marcus Appleby, Tuan-Yow Chien, Steven Flammia, and Shayne Waldron.
\newblock Constructing exact symmetric informationally complete measurements
  from numerical solutions.
\newblock {\em J. Phys. A}, 51(16):165302, 40, 2018.

\bibitem[BADL24]{BSDL24}
Santiago Barrera~Acevedo, Heiko Dietrich, and Corey Lionis.
\newblock New families of quaternionic {H}adamard matrices.
\newblock {\em Des. Codes Cryptogr.}, 92(9):2511--2525, 2024.

\bibitem[Bli17]{B17}
Hans~Frederik Blichfeldt.
\newblock {\em Finite collineation groups : with an introduction to the theory
  of operators and substitution groups.}
\newblock University of Chicago science series. University of Chicago Press,
  Chicago, 1917.

\bibitem[BST23]{BST23}
Gwyn Bellamy, Johannes Schmitt, and Ulrich Thiel.
\newblock On parabolic subgroups of symplectic reflection groups.
\newblock {\em Glasg. Math. J.}, 65(2):401--413, 2023.

\bibitem[BW13]{BW13}
Helen Broome and Shayne Waldron.
\newblock On the construction of highly symmetric tight frames and complex
  polytopes.
\newblock {\em Linear Algebra Appl.}, 439(12):4135--4151, 2013.

\bibitem[CKM16]{CKM16}
Henry Cohn, Abhinav Kumar, and Gregory Minton.
\newblock Optimal simplices and codes in projective spaces.
\newblock {\em Geom. Topol.}, 20(3):1289--1357, 2016.

\bibitem[Coh80]{C80}
Arjeh~M. Cohen.
\newblock Finite quaternionic reflection groups.
\newblock {\em J. Algebra}, 64(2):293--324, 1980.

\bibitem[Coh91]{C90}
A.~M. Cohen.
\newblock Presentations for certain finite quaternionic reflection groups.
\newblock In {\em Advances in finite geometries and designs ({C}helwood {G}ate,
  1990)}, Oxford Sci. Publ., pages 69--79. Oxford Univ. Press, New York, 1991.

\bibitem[Cro59]{Cr59}
Donald~W. Crowe.
\newblock A regular quaternion polygon.
\newblock {\em Canad. Math. Bull.}, 2:77--79, 1959.

\bibitem[CS03]{CS03}
John~H. Conway and Derek~A. Smith.
\newblock {\em On quaternions and octonions: their geometry, arithmetic, and
  symmetry}.
\newblock A K Peters, Ltd., Natick, MA, 2003.

\bibitem[DGS77]{DGS77}
P.~Delsarte, J.~M. Goethals, and J.~J. Seidel.
\newblock Spherical codes and designs.
\newblock {\em Geometriae Dedicata}, 6(3):363--388, 1977.

\bibitem[FWZ11]{FWZ11}
F.~O. Farid, Qing-Wen Wang, and Fuzhen Zhang.
\newblock On the eigenvalues of quaternion matrices.
\newblock {\em Linear Multilinear Algebra}, 59(4):451--473, 2011.

\bibitem[Gan11]{G11}
Iordan Ganev.
\newblock Real and quaternionic representations of finite groups.
\newblock 2011.
\newblock \url{https://ivganev.github.io/notes/files/RealQuatRepns.pdf}.

\bibitem[Gan25]{G22}
Mikhail Ganzhinov.
\newblock Highly symmetric lines.
\newblock {\em Linear Algebra Appl.}, 722:12--37, 2025.

\bibitem[Hog78]{H79}
S.~G. Hoggar.
\newblock Bounds for quaternionic line systems and reflection groups.
\newblock {\em Math. Scand.}, 43(2):241--249 (1979), 1978.

\bibitem[Hog82]{H82}
S.~G. Hoggar.
\newblock {$t$}-designs in projective spaces.
\newblock {\em European J. Combin.}, 3(3):233--254, 1982.

\bibitem[Hog84]{H84}
S.~G. Hoggar.
\newblock Parameters of {$t$}-designs in {${\bf F}P^{d-1}$}.
\newblock {\em European J. Combin.}, 5(1):29--36, 1984.

\bibitem[IJM20]{IJM20}
Joseph~W. Iverson, John Jasper, and Dustin~G. Mixon.
\newblock Optimal line packings from finite group actions.
\newblock {\em Forum Math. Sigma}, 8:Paper No. e6, 40, 2020.

\bibitem[Iva81]{I81}
I.~D. Ivanovi\'{c}.
\newblock Geometrical description of quantal state determination.
\newblock {\em J. Phys. A}, 14(12):3241--3245, 1981.

\bibitem[Kan95]{K95}
William~M. Kantor.
\newblock Quaternionic line-sets and quaternionic {K}erdock codes.
\newblock {\em Linear Algebra Appl.}, 226/228:749--779, 1995.

\bibitem[LT09]{LT09}
Gustav~I. Lehrer and Donald~E. Taylor.
\newblock {\em Unitary reflection groups}, volume~20 of {\em Australian
  Mathematical Society Lecture Series}.
\newblock Cambridge University Press, Cambridge, 2009.

\bibitem[MW24]{MW24}
Daniel McNulty and Stefan Weigert.
\newblock Mutually unbiased bases in composite dimensions -- a review, 2024.

\bibitem[Sch23]{S23}
Johannes Schmitt.
\newblock {\em On Q-factorial terminalizations of symplectic linear quotient
  singularities}.
\newblock doctoralthesis, Rheinland-Pf{\"a}lzische Technische Universit{\"a}t
  Kaiserslautern-Landau, 2023.

\bibitem[SS95]{SS95}
G.~Scolarici and L.~Solombrino.
\newblock Notes on quaternionic group representations.
\newblock {\em Internat. J. Theoret. Phys.}, 34(12):2491--2500, 1995.

\bibitem[ST54]{ST54}
G.~C. Shephard and J.~A. Todd.
\newblock Finite unitary reflection groups.
\newblock {\em Canad. J. Math.}, 6:274--304, 1954.

\bibitem[Voi21]{V21}
John Voight.
\newblock {\em Quaternion algebras}, volume 288 of {\em Graduate Texts in
  Mathematics}.
\newblock Springer, Cham, [2021] \copyright 2021.

\bibitem[Wal18]{W18}
Shayne F.~D. Waldron.
\newblock {\em An introduction to finite tight frames}.
\newblock Applied and Numerical Harmonic Analysis. Birkh\"{a}user/Springer, New
  York, 2018.

\bibitem[Wal20]{W20c}
Shayne Waldron.
\newblock A variational characterisation of projective spherical designs over
  the quaternions, 2020.

\bibitem[Wal24]{W24}
Shayne Waldron.
\newblock The geometry of the six quaternionic equiangular lines in
  $\mathbb{H}^2$, 2024.

\bibitem[Wal25]{W25}
Shayne Waldron.
\newblock An elementary classification of the quaternionic reflection groups of
  rank two, 2025.

\bibitem[WF89]{WF89}
William~K. Wootters and Brian~D. Fields.
\newblock Optimal state-determination by mutually unbiased measurements.
\newblock {\em Ann. Physics}, 191(2):363--381, 1989.

\bibitem[Zha97]{Z97}
Fuzhen Zhang.
\newblock Quaternions and matrices of quaternions.
\newblock {\em Linear Algebra Appl.}, 251:21--57, 1997.

\end{thebibliography}
\nocite{*}

\end{document}